\documentclass[letterpaper, 10 pt, conference]{ieeeconf}  

\IEEEoverridecommandlockouts                              
\usepackage{amsmath,amssymb}
\usepackage{stmaryrd}
\usepackage{mathtools}
\usepackage{tikz}
\usetikzlibrary{calc}
\usepackage{nicefrac}
\usepackage[algo2e, ruled, vlined, linesnumbered]{algorithm2e}
\usepackage{setspace}
\usepackage[normalem]{ulem}
\usepackage[scr=esstix]{mathalpha}
\usepackage[T1]{fontenc}
\usepackage{array}
\usepackage{nicematrix}
\usepackage{multirow}
\usepackage{booktabs}
\NiceMatrixOptions{
    code-for-first-row = \color{gray} ,
    code-for-last-row = \color{gray} ,
    code-for-first-col = \color{gray} ,
    code-for-last-col = \color{gray}
}
\usepackage{inconsolata}
\usepackage{subcaption}
\usepackage{cancel}

\usepackage{etoolbox}
\makeatletter
\patchcmd\algocf@Vline{\vrule}{\vrule \kern-0.4pt}{}{}
\patchcmd\algocf@Vsline{\vrule}{\vrule \kern-0.4pt}{}{}
\makeatother

\usepackage{hyperref}
\usepackage{cleveref}

\makeatletter
\newcommand{\verbatimfont}[1]{\renewcommand{\verbatim@font}{\ttfamily#1}}
\makeatother

\makeatletter
\patchcmd{\@algocf@start}{-1.5em}{0em}{}{}
\makeatother

\SetAlgoNlRelativeSize{-1} 
\SetAlgoSkip{}
\SetKwInOut{KwIn}{In}
\SetKwInOut{KwOut}{Out}
\SetKwProg{Fn}{Function}{}{}
\SetKwFor{While}{while}{}{}%
\SetKwFor{GotoLoop}{}{}{}%
\SetKwFor{For}{for}{}{}%
\SetKwIF{If}{ElseIf}{Else}{if}{}{else if}{else}{end if}%
\newcommand\mykwcomment{\quad$\smalltriangleright$\ }
\SetKwComment{Comment}{\mykwcomment}{}
\SetKwComment{FullLineComment}{$\smalltriangleright$\ }{}

\SetCommentSty{mycommentstyle}
\newcommand\mycomment[1]{\hfill\Comment{#1}}

\newtheorem{corollary}{Corollary}
\newtheorem{proposition}{Proposition}
\newtheorem{lemma}{Lemma}

\newtheorem{definition}{Definition}

\newtheorem{remark}{Remark}

\crefname{theorem}{theorem}{theorems}
\crefname{corollary}{corollary}{corollaries}
\crefname{proposition}{proposition}{propositions}
\crefname{lemma}{lemma}{lemmas}
\crefname{assumption}{assumptions}{assumptions}
\crefname{definition}{definition}{definitions}
\crefname{example}{example}{examples}
\crefname{remark}{remark}{remarks}

\Crefname{theorem}{Theorem}{Theorems}
\Crefname{corollary}{Corollary}{Corollaries}
\Crefname{proposition}{Proposition}{Propositions}
\Crefname{lemma}{Lemma}{Lemmas}
\Crefname{assumption}{Assumption}{Assumptions}
\Crefname{definition}{Definition}{Definitions}
\Crefname{example}{Example}{Examples}
\Crefname{remark}{Remark}{Remarks}

\newcommand\defeq{\triangleq}
\newcommand\rlddots{\mathinner{
    \kern1mu\raise1pt{.}
    \kern2mu\raise4pt{.}
    \kern2mu\raise7pt{\Rule{0pt}{7pt}{0pt}.}
    \kern1mu
}}

\newcommand\posdef{\succ}
\newcommand\negdef{\prec}

\newcommand\lt{<}
\newcommand\gt{>}

\newcommand\hess{\nabla^2}
\newcommand\Reals{\mathrm{I\hspace{-2pt}R}}

\newcommand\Diag{\mathbb{D}}
\newcommand\Lower{\mathbb{L}}
\newcommand\Upper{\mathbb{U}}
\newcommand\R{\Reals}

\newcommand\minimize{\operatorname*{\mathbf{minimize}}}

\newcommand\subjto{\operatorname*{\mathbf{subject\;to}}}

\newcommand\I{\mathrm{I}}

\newcommand\diag{\operatorname{\mathbf{diag}}}

\newcommand\inv[1]{#1^{-\hspace{-1pt}1}}

\newcommand\tp[1]{#1^\top}

\newcommand\ttp[1]{#1^{\!\top\!}}

\newcommand\landauO{\mathcal{O}}

\newcommand\hookuparrow{%
  \mathchoice
    {\scalebox{0.86}{\rotatebox[origin=c]{90}{$\displaystyle\hookrightarrow$}}}
    {\scalebox{0.86}{\rotatebox[origin=c]{90}{$\textstyle\hookrightarrow$}}}
    {\scalebox{0.86}{\rotatebox[origin=c]{90}{$\scriptstyle\hookrightarrow$}}}
    {\scalebox{0.86}{\rotatebox[origin=c]{90}{$\scriptscriptstyle\hookrightarrow$}}}
}
\newcommand\assign{\leftarrow}

\newcommand\setJ{\mathcal{J}}
\newcommand\Upsx[1]{\Upsilon_{\hspace{-1.5pt}x\hspace{-1pt},#1}}
\newcommand\Upsu[1]{\Upsilon_{\hspace{-1.5pt}u\hspace{-0.5pt},#1}}
\newcommand\Qx[1]{\breve Q_{\hspace{-0.1pt}x\hspace{-1pt},#1\hspace{0.5pt}}}
\newcommand\Qu[1]{\breve Q_{\hspace{-0.2pt}u\hspace{-0.5pt},#1}}
\newcommand\LPxx[1]{L_{x\hspace{-1pt}x\hspace{-1pt},#1}}
\newcommand\LPxu[1]{L_{x\hspace{-1pt}u\hspace{-0.5pt},#1}}
\newcommand\LPuu[1]{L_{u\hspace{-1pt}u\hspace{-0.5pt},#1}}
\newcommand\LPTxx[1]{\tp L_{x\hspace{-1pt}x\hspace{-1pt},#1}}
\newcommand\tilLPxx[1]{{\tilde L}_{x\hspace{-1pt}x\hspace{-1pt},#1}}
\newcommand\tilLPxu[1]{{\tilde L}_{x\hspace{-1pt}u\hspace{-0.5pt},#1}}
\newcommand\tilLPuu[1]{{\tilde L}_{u\hspace{-1pt}u\hspace{-0.5pt},#1}}
\newcommand\tilLPTxx[1]{\tp {\tilde L}_{x\hspace{-1pt}x\hspace{-1pt},#1}}

\newcommand\Ssign{\mathscr S}
\newcommand\smallonesigma{\left(\! \begin{smallmatrix} 1 \\ & \Sigma \end{smallmatrix} \!\right)}

\newcommand\Cpp{\texttt{C+\hspace{-0.5pt}+}}

\newcommand\smalltriangleright{\triangleright}

\newcommand\QPALMOCP{\textsc{qpalm-ocp}}
\newcommand\QPALM{\textsc{qpalm}}

\makeatletter

\newdimen\@IEEEtmpitemindent
\def\@begintheorem#1#2{\@IEEEtmpitemindent\itemindent\topsep 3pt\rmfamily\trivlist%
    \item[\hskip \labelsep{\indent\itshape #1\ #2:}]\itemindent\@IEEEtmpitemindent}
\def\@opargbegintheorem#1#2#3{\@IEEEtmpitemindent\itemindent\topsep 5pt\rmfamily \trivlist%
    \item[\hskip\labelsep{\indent\itshape #1\ #2\ (#3):}]\itemindent\@IEEEtmpitemindent}
\def\@endtheorem{\endtrivlist\vskip 3pt\unskip}
\makeatother

\title{\LARGE \bf
Blocked Cholesky factorization updates
of the Riccati recursion using hyperbolic Householder transformations
}

\author{Pieter Pas \quad and \quad Panagiotis Patrinos
\thanks{This work was supported by
FWO PhD grant 11M9523N; FWO projects G081222N, G0A0920N, G033822N;
KU Leuven internal funding C14/24/103, C14/18/068, 12Y7622N;
the European Union's Horizon 2020 research and innovation programme under the Marie Skłodowska-Curie grant agreement No. 953348;
and EuroHPC Project 101118139 Inno4Scale.
}
\thanks{The authors are with Department of Electrical Engineering (ESAT), STADIUS Center for
Dynamical Systems, Signal Processing and Data Analytics,
KU Leuven, Kasteelpark Arenberg 10, 3001 Leuven, Belgium. (e-mail: \texttt{\{pieter.pas,panos.patrinos\}@esat.kuleuven.be})}%
}

\begin{document}

\maketitle
\thispagestyle{empty}
\pagestyle{empty}

\begin{abstract}
Newton systems in quadratic programming (QP) methods are often solved using direct
Cholesky or $LD\ttp L$ factorizations. When the linear systems in successive
iterations differ by a
low-rank modification (as is common in active set and augmented Lagrangian
methods), updating the existing factorization can offer significant
performance improvements over recomputing a full Cholesky factorization.
We review the hyperbolic Householder transformation, and demonstrate its
usefulness in describing low-rank Cholesky factorization updates.
By applying this hyperbolic Householder-based framework to the well-known
Riccati recursion for solving saddle-point problems with optimal control structure,
we develop a novel algorithm for updating the factorizations used in optimization solvers for
optimal control. Specifically, the proposed method can be used to
efficiently solve the semismooth Newton systems that are at the core of the
augmented Lagrangian-based QPALM-OCP solver. An optimized open-source implementation
of the proposed factorization update routines is provided as well.
\end{abstract}

\section{Introduction}

In many types of numerical optimization solvers,
the most computationally demanding steps are the construction and factorization of the
linear system that determines a Newton step.
In practice, such a system could be solved using a Cholesky decomposition $L\ttp L$ or $LD\ttp L$,
where $L$ is lower triangular and $D$ diagonal. When two Newton systems in
successive iterations of the solver are similar, it may be advantageous to
\textit{update} the existing factorization from the previous iteration, rather
than to perform a new Cholesky decomposition in the current iterate.
Active set solvers are one prominent class of methods where this is the case:
from one iteration to the next, a single constraint (or a small number of constraints)
is added to or removed from the active set, and this translates to the addition of a low-rank term
in the Newton system \cite{wright_applying_1996}. Techniques for computing the updated Cholesky factor $\tilde L$ of the
matrix $L\ttp L + \sigma a\ttp a$ (for a known triangular factor $L$ and a rank-one update by a vector $a$) are well known:
an early overview of various methods was given by Gill et al. in \cite{gill_methods_1974}.
Their factorization update methods have seen widespread use in active set \cite{kirches_factorization_2011,nielsen_low-rank_2018},
augmented Lagrangian \cite{hermans_qpalm_2022,lowenstein_qpalm-ocp_2024} and even interior point solvers \cite{frey_active-set_2020}.
Since a factorization update for a rank-one modification of an $n$-by-$n$ matrix can be computed in $\landauO(n^2)$
operations as opposed to $\landauO(n^3)$ operations for a full Cholesky factorization, using them effectively can significantly improve
solver run times.

An alternative framework for describing factorization updates is based on
hyperbolic Householder transformations, a generalization of the standard orthogonal
Householder transformations often used to compute $QR$ factorizations.
This framework was originally developed by Rader and Steinhardt for adding and removing data in least squares
problems for signal processing applications \cite{rader_hyperbolic_1986}.
Subsequent studies analyzed the numerical properties of hyperbolic Householder transformations
for the problem of updating Cholesky factorizations \cite{bojanczyk_stability_1991,stewart_hyperbolic_1998}.
This Householder-based approach to factorization updates naturally
generalizes to the case where multiple rank-one matrices are added or subtracted
at once. That is, it provides an intuitive way of representing the low-rank
updating problem
\begin{equation} \label{eq:fact-update-intro}
    \tilde L \ttp{\tilde L} = L\ttp L + A\Sigma \ttp A,
\end{equation}
where we wish to find the lower triangular Cholesky factor $\tilde L$
of the matrix $L\ttp L + A\Sigma \ttp A \in \R^{n\times n}$, for some known Cholesky decomposition $L\ttp L$
and a low-rank modification by the update matrix $A\in \R^{n\times m}$ and a
diagonal matrix $\Sigma$. We will refer to this problem simply as the \textit{updating}
problem, even though the specific case where $\Sigma \negdef 0$ is sometimes referred to as a \textit{downdating} problem.
Factorization updates of the form \eqref{eq:fact-update-intro} are of interest to us,
because they match exactly the type of low-rank modifications that are encountered
in the inner semismooth Newton method at the core of the
augmented Lagrangian-based \QPALM{} solver \cite[\S4.1.2]{hermans_qpalm_2022}.

The hyperbolic Householder framework has two main advantages over the methods from \cite{gill_methods_1974}:
1. It simplifies the notation in theoretical derivations, providing a convenient mental model
for reasoning about factorization updates;
and 2. by merging multiple rank-one updates in a single operation,
it enables high-performance practical implementations.
While such blocked implementations built on top of BLAS have been shown to scale well for very large matrices \cite{van_de_geijn_high-performance_2011},
optimized versions for smaller matrices (such as the ones used in optimal control)
are not yet available.

\subsection*{Contributions}

The goal of the present paper is threefold:
\begin{itemize}
    \item Provide a self-contained overview of hyperbolic Householder transformations,
    specifically in the context of the Cholesky factorization updating problem \eqref{eq:fact-update-intro}.
    (\S\ref{sec:householder}, \S\ref{sec:algorithms})
    \item Present a practical use case for the hyperbolic Householder framework,
    using it to derive a novel method for updating the Cholesky factorizations
    in a solver for saddle-point problems with optimal control structure,
    based on the well-known Riccati recursion. (\S\ref{sec:riccati})
    \item Demonstrate the practical performance of hyperbolic Householder operations
    using optimized, open-source \Cpp{} implementations of the presented algorithms. (\S\ref{sec:performance})
\end{itemize}
The Riccati-based factorization update routines, along with their efficient
implementation, offer a compelling approach for solving Newton systems in the
recently proposed \QPALMOCP{} solver for linear-quadratic optimal control problems \cite{lowenstein_qpalm-ocp_2024}.

\subsection*{Notation}

We use $\Lower$, $\Upper$ and $\Diag$ to denote the sets of lower triangular, upper triangular
and diagonal matrices respectively.
The symbol $\mathbf{e}_1$ denotes the first standard basis vector $(1\; 0\; ...\; 0)$.
For vectors, $x_i$ refers to the $i$-th element of $x$.
For block matrices, $A_{ij}$ refers to the subblock in block row $i$ and block column $j$ of $A$.
In pseudocode, an empty matrix $(\,)$ represents the neutral element for block matrix concatenation.
We use tildes to denote ``updated'' matrices.
Subscripts in parentheses denote the iteration index in an algorithm,
and in the context of optimal control,
a superscript (for vectors) or subscript (for matrices) is used for the stage index.

\section{Hyperbolic Householder transformations} \label{sec:householder}

\subsection{Construction and properties}

\begin{definition}[$\Ssign$-orthogonality] \label{def:s-orth}
    Given a diagonal matrix $\Ssign \in \Diag(\R^n)$,
    we say that a matrix $\breve Q \in \R^{n\times n}$ is $\Ssign$-orthogonal if
    \begin{equation} \label{eq:s-orthogonality}
        \breve Q \Ssign \ttp{\breve Q} = \Ssign.
    \end{equation}
\end{definition}
In the case where $\Ssign = \I$, \Cref{def:s-orth} reduces to the definition of standard orthogonality, $\breve Q\tp{\breve Q}=\I$.
It is easy to show that if $\breve Q_{(1)}$ and $\breve Q_{(2)}$ are both $\Ssign$-orthogonal, then their product
$\breve Q_{(1)} \breve Q_{(2)}$ is also $\Ssign$-orthogonal.

An important application of standard Householder transformations is the
introduction of zeros during e.g. QR factorization. 
Additionally, their orthogonality allows efficient
application of their inverses.
The following lemma shows that a generalized \textit{hyperbolic}
Householder transformation can be constructed, with properties similar to the
standard Householder transformation:
A hyperbolic Householder matrix is $\Ssign$-orthogonal, and can be constructed
to reduce all but the first element of a given vector to zero.
\begin{lemma}[Hyperbolic Householder reflectors]\label{lem:hyper-householder}
    Consider a matrix $\Ssign \defeq \smallonesigma \in \Diag(\R^{n})$ and a vector $x \in \R^{n}$
    with $\ttp x \Ssign x \gt 0$.
    Then the matrix $\breve Q \in \R^{n\times n}$ defined
    below is $\Ssign$-orthogonal, and satisfies $\ttp {\breve Q} x = \eta \mathbf{e}_1$.
    \begin{equation} \label{eq:def-q-hyper-house-lem}
        \breve Q \defeq \I - \frac{2}{\ttp v\Ssign v}\, \Ssign v \ttp v, \hspace{0.5em}
        \begin{aligned}
            &\text{where}\hspace{-0.7em}&
            v &\defeq x - \eta \mathbf{e}_1, \\[-0.4em]
            &\text{and} & \eta &\defeq -\operatorname{sign}(x_1)\sqrt{\ttp x \Ssign x}.\,\footnotemark
        \end{aligned}
        \footnotetext{Using $\operatorname{sign}(0) = 1$.}
    \end{equation}
\end{lemma}
\begin{proof}
    The $\Ssign$-orthogonality of $\breve Q$ and the value of $\ttp{\breve Q} x$
    follow from straightforward algebraic manipulations. \\[0.2em]
    $\phantom{ll}
            \breve Q \Ssign \ttp{\breve Q}
            = \Ssign - {\tfrac{2\Ssign v \ttp v \Ssign}{\ttp v\Ssign v}}
            - {\tfrac{2\Ssign v \ttp v \Ssign}{\ttp v\Ssign v}} +
            {\tfrac{2\Ssign v}{\ttp v\Ssign v}{
            (\ttp v\Ssign v)
            \tfrac{2\ttp v \Ssign}{\ttp v\Ssign v}}} = \Ssign.
    $ \\[0.2em]
    Since $\ttp v \Ssign v = \ttp x \Ssign x - 2 \ttp x \Ssign \eta \mathbf{e}_1 + \eta^2 = 2 \eta^2 - 2 x_1 \eta = -2 \eta v_1$, \\[0.2em]
    $\phantom{m}
            \ttp {\breve Q} x
            = x - \tfrac{2\ttp v \Ssign x}{\ttp v\Ssign v} v  
            = x - \tfrac{2\ttp v \Ssign (v + \eta \mathbf{e}_1)}{\ttp v\Ssign v} v 
            = x - v = \eta \mathbf{e}_1.
    $ \\[0.2em]
    Finally, the choice of the sign of $\eta$ ensures that $-2x_1\eta \ge 0$, thus $\ttp v \Ssign v \gt 0$, and $\breve Q$ is well-defined.
\end{proof}
\begin{remark}[Hyperbolic Householder matrix construction]
    Even without knowing $\breve Q$, its $\Ssign$-orthogonality and the fact
    that only the first component of $\ttp{\breve Q}x$ is nonzero are sufficient to
    determine (up to a sign) the value of this component $\eta$:
    $\ttp x \Ssign x = \ttp x \breve Q \Ssign \ttp {\breve Q} x = \eta^2$.
    The construction of $\breve Q$ in \Cref{lem:hyper-householder} is based on that of the classical Householder reflector
    \cite[Fig.~10.1]{trefethen_numerical_1997}, generalized for quadratic forms using $\Ssign$ rather than the Euclidean inner product.
\end{remark}
\begin{remark}[Alternative definitions] \label{rem:sign-q}
    Rader and Steinhardt originally defined the following symmetric hyperbolic Householder matrix
    $
        \breve Q_{\mathrm{RS}} = \Ssign - \tfrac{2}{\ttp v \Ssign v} v \ttp v
    $
    \cite[Eq.~28]{rader_hyperbolic_1986}.
    When $\Ssign$ consists only of diagonal elements $\pm 1$, which
    is the case considered in \cite{rader_hyperbolic_1986}, their definition is related to
    \eqref{eq:def-q-hyper-house-lem} through the identity $\breve Q = \Ssign \breve Q_{\mathrm{RS}}$, which can be seen using $\Ssign^2 = \I$.
    The definition in \eqref{eq:def-q-hyper-house-lem} is inspired by the one used in \cite[\S3]{stewart_hyperbolic_1998},
    and even though symmetry is lost, it has the advantage of generalizing to matrices $\Ssign$ with entries different from $\pm1$.
    We note that $\breve Q$ and $\eta$ in \Cref{lem:hyper-householder} are unique only up to some signs, and similar results can be formulated if $\breve Q$ is multiplied by an $\Ssign$-orthogonal matrix $\diag({}\pm\!1\; \dots\, {}\pm\!1)$.
\end{remark}

The hyperbolic Householder transformation introduced in \Cref{lem:hyper-householder}
serves as a building block for an $\Ssign$-orthogonal generalization of the $QR$
factorization or its transpose, the $LQ$ factorization.
The $L\breve Q$ factorization described by the following proposition (with 
$L$ lower triangular and $\breve Q$ $\Ssign$-orthogonal)
enables $\Ssign$-orthogonal triangularization, a key step in performing
Cholesky factorization updates.
\begin{proposition}[Hyperbolic $L \breve Q$ factorization]
    \label{prop:hyper-qr}
    Consider a matrix $\begin{pmatrix}
        L & A
    \end{pmatrix}$ with $L \in \Lower(\R^{n\times n})$ and $A \in \R^{n\times m}$,
    and a diagonal matrix $\Sigma \in \Diag(\R^m)$ for which $L\ttp L + A\Sigma \ttp A \posdef 0$.
    Then there exists an $\Ssign$-orthogonal matrix $\breve Q \in \R^{(n+m)\times (n+m)}$
    with $\Ssign = \left(\! \begin{smallmatrix} \I_n \\ & \Sigma \end{smallmatrix} \!\right)$
    such that
    \begin{equation}
        \begin{pmatrix}
            L & A
        \end{pmatrix} \breve Q = \begin{pmatrix}
            \tilde L & 0
        \end{pmatrix},
    \end{equation}
    with $\tilde L \in \Lower(\R^{n\times n})$ lower triangular.
\end{proposition}
\begin{proof}
    \renewcommand*{\arraystretch}{1}
    We will recursively construct $\breve Q$ as a product of $\Ssign$-orthogonal
    transformations, each annihilating one row of $A$.
    The base case for size $n=1$ is covered by \Cref{lem:hyper-householder}. For $n>1$, partition the given matrices as follows:
    \begin{equation} \label{eq:proof-hyper-qr-partition}
        \left(\begin{array}{@{}c|c|c@{}}
            \lambda & & \ttp a \\\hline
            \ell & L' & A'
        \end{array}\right) \defeq \begin{pmatrix}
            L & A
        \end{pmatrix},
        \hspace{0.5em}
        \begin{aligned}
            &\text{with } \lambda \in \R, \ell \in \R^{n-1}, \\[-0.3em]
            &L' \in \Lower(\R^{(n-1)\times (n-1)}), \\[-0.4em]
            &a \in \R^{m}, A' \in \R^{(n-1)\times m}.
        \end{aligned}\hspace{-0.4em}
    \end{equation}
    From the assumption $\tilde H \defeq L\ttp L + A\Sigma\ttp A \posdef 0$, we know that the top left element of $\tilde H$ is strictly positive \cite[Prop.~8.2.8]{bernstein_matrix_2009}, i.e.
    $\ttp x \smallonesigma x > 0$ with $x \defeq \left( \begin{smallmatrix}
        \lambda \\ a
    \end{smallmatrix} \right)$. We can therefore invoke \Cref{lem:hyper-householder} to construct a
    $\smallonesigma$-orthogonal
    matrix $\breve Q_{(n)}$ that annihilates $a$, i.e.
    $
        \left(\begin{array}{@{}cc@{}}
            \lambda & \ttp a
        \end{array}\right) \breve Q_{(n)} =
        \left(\begin{array}{@{}cc@{}}
            \tilde\lambda & 0
        \end{array}\right),
    $ with $\tilde \lambda = -\operatorname{sign}(\lambda)\sqrt{\lambda^2 + \ttp a \Sigma a}$.
    Partition $\breve Q_{(n)}$ by isolating its first row and column, and embed it
    into a larger matrix $\breve Q^{\hookuparrow}_{(n)} \in \R^{(n+m)\times(n+m)}$ which is clearly $\Ssign$-orthogonal:
    \begin{equation*}
        \begin{aligned}
            \breve Q^{\hookuparrow}_{(n)} &\defeq
        \scalebox{0.85}{\setstretch{1.2}$
        \begin{pmatrix}
            \breve q_{(n),11} & & \tp {\breve q}_{(n),12} \\
            & \I_{n-1} & \\
            \breve q_{(n),21} & & \breve Q_{(n),22} \\
        \end{pmatrix}$}, &
        \begin{aligned}
            &\text{with } \breve q_{(n),11} \in \R, \\[-0.2em]
            &\breve q_{(n),12},\, \breve q_{(n),21} \in \R^m.
        \end{aligned}
    \end{aligned}
    \end{equation*}
    Defining $\begin{pmatrix}
        \tilde \ell & \tilde A'
    \end{pmatrix} \defeq \begin{pmatrix}
        \ell & A'
    \end{pmatrix} \breve Q_{(n)}$, we have that
    \begin{equation} \label{eq:hyper-qr-first-row-annihilated}
        \renewcommand*{\arraystretch}{1.2}
        \begin{pmatrix}
            L & A
        \end{pmatrix} \breve Q^{\hookuparrow}_{(n)} = \scalebox{0.85}{$\left(\begin{array}{@{}c|c|c@{}}
            \tilde\lambda & & \\\hline
            \tilde\ell & L' & \tilde A'
        \end{array}\right)$}.
    \end{equation}
    $\breve Q_{(n)}$ reduced all nondiagonal nonzeros in the top row to zero, and computed the first column of $\tilde L$.
    To proceed, we will assume inductively that \Cref{prop:hyper-qr} holds for size $n-1$,
    so that we can use it to annihilate $\tilde A'$ and to update $L'$. To apply the proposition, we require $\tilde H' \defeq L' L^{\prime\top} + \tilde A' \Sigma \tilde A^{\prime\top} \posdef 0$.
    This holds, because $\tilde H'$ is the Schur complement of $\tilde\lambda^2$ in
    \begin{equation} \label{eq:proof-prop-schur-complement}
        \renewcommand*{\arraystretch}{1.2}
        \begin{aligned}
            &\phantom{=}
            \scalebox{0.85}{$\left(\begin{array}{@{}c|c@{}}
                \tilde\lambda^2 & \tilde\lambda \tp{\tilde\ell} \\\hline
                \tilde\lambda\tilde\ell & 
                \tilde\ell \tp{\tilde\ell} + \tilde H'
            \end{array}\right)$} = 
            \scalebox{0.85}{$\left(\begin{array}{@{}c|c|c@{}}
                \tilde\lambda & & \\\hline
                \tilde\ell & L' & \tilde A'
            \end{array}\right)$} \Ssign \scalebox{0.85}{$\left(\begin{array}{@{}c|c|c@{}}
                \tilde\lambda & & \\\hline
                \tilde\ell & L' & \tilde A'
            \end{array}\right)$}^{\!\!\top} \\
            &\overset{\eqref{eq:hyper-qr-first-row-annihilated}}=\;
            \begin{pmatrix}
                L & A
            \end{pmatrix}
            \breve Q^{\hookuparrow}_{(n)} \Ssign \breve Q^{\hookuparrow\top}_{(n)}
            \ttp{\begin{pmatrix}
                L & A
            \end{pmatrix}} = \tilde H,
        \end{aligned}
    \end{equation}
    where we used the $\Ssign$-orthogonality of $\breve Q^{\hookuparrow}_{(n)}$ in the last step.
    Since $\tilde H \posdef 0$ and $\tilde\lambda^2>0$, $\tilde H' = \tilde H / \tilde\lambda^2 \posdef 0$
    \cite[Thm.~7.7.7]{horn_matrix_2012}.
    Recursive application of the proposition then provides us with an $\left(\!\begin{smallmatrix}
        \I_{n-1} \\ & \Sigma
    \end{smallmatrix}\!\right)$-orthogonal matrix $\breve Q'$ that satisfies
    \begin{equation} \label{eq:hyper-qr-bottom-right-annihilated}
        \begin{pmatrix}
            L' & \tilde A'
        \end{pmatrix} \breve Q' = \begin{pmatrix}
            \tilde L' & 0
        \end{pmatrix}.
    \end{equation}
    The matrix $\breve Q'$ can again be embedded into a larger $\Ssign$-orthogonal matrix $\breve Q^{\hookuparrow}_{(n-1,1)} \defeq
    \left(\!\begin{smallmatrix} 1 \\ & \breve Q' \end{smallmatrix}\!\right)$.
    Therefore,
    \begin{equation} \label{eq:product-q-proof-qr}
        \renewcommand*{\arraystretch}{1.2}
        \begin{aligned}
            \begin{pmatrix} L & A \end{pmatrix}
            \underbrace{\breve Q^{\hookuparrow}_{(n)}\breve Q^{\hookuparrow}_{(n-1,1)}}_{\mathclap{\defeq \breve Q}}
                \;&\overset{\eqref{eq:hyper-qr-first-row-annihilated}}=\;
            \scalebox{0.85}{$\left(\begin{array}{@{}c|c|c@{}}
                \tilde\lambda & & \\\hline
                \tilde\ell & L' & \tilde A'
            \end{array}\right)$} \scalebox{0.85}{$\begin{pmatrix}
                1 \\ & \breve Q'
            \end{pmatrix}$} \\[-0.9em]
            &\overset{\eqref{eq:hyper-qr-bottom-right-annihilated}}=\;
            \scalebox{0.85}{$\left(\begin{array}{@{}c|c|c@{}}
                \tilde\lambda & & \\\hline
                \tilde\ell & \tilde L' & \mathclap{\hspace{7pt}0}\phantom{\tilde A'}
            \end{array}\right)$}.
        \end{aligned}
    \end{equation}
    Defining
    $\tilde L\defeq \left( \!\begin{smallmatrix}
        \tilde\lambda \\ \tilde\ell & \tilde L'
    \end{smallmatrix} \!\right)$ concludes the proof.
\end{proof}

\begin{corollary}[Factorization updates] \label{cor:fact-update} \strut\\
    Consider a Cholesky factor $L\in \Lower(\R^{n\times n})$ of the symmetric positive definite matrix
    $H \defeq L\ttp L$.
    Let $\tilde H \defeq H + A \Sigma \ttp A$ be positive definite
    for a given update matrix $A \in \R^{n\times m}$ and a diagonal matrix $\Sigma \in \Diag(\R^{m})$.
    Then there exists an $\Ssign$-orthogonal matrix $\breve Q$
    with $\Ssign \defeq \left(\!\begin{smallmatrix}
        \I_n \\ & \Sigma
    \end{smallmatrix}\!\right)$,
    such that the matrix $\tilde L \in \Lower(\R^{n\times n})$ defined below is a Cholesky factor of $\tilde H$.
    \begin{equation} \label{eq:factup-prop}
        \begin{pmatrix}
            \tilde L & 0
        \end{pmatrix}
        \defeq
        \begin{pmatrix}
            L & A
        \end{pmatrix} \breve Q.
    \end{equation}
\end{corollary}
\begin{proof}
    \Cref{prop:hyper-qr} ensures existence of such $\breve Q$.
    Multiplying both sides of \eqref{eq:factup-prop} by $\Ssign$ and by
    their own transposes:
    \begin{equation}
        \begin{aligned}[t]
            \begin{pmatrix}
                \tilde L & 0
            \end{pmatrix} \Ssign
            \ttp {\begin{pmatrix}
                \tilde L & 0
            \end{pmatrix}}
            &=
            \begin{pmatrix}
                L & A
            \end{pmatrix} \breve Q \Ssign \ttp {\breve Q}
            \ttp{\begin{pmatrix}
                L & A
            \end{pmatrix}} \\[-0.2em]
            \overset{\eqref{eq:s-orthogonality}}{\Leftrightarrow} \quad
            \tilde L \ttp {\tilde L}
            &=
            \begin{pmatrix}
                L & A
            \end{pmatrix}
            \begin{pmatrix}
                \!\begin{smallmatrix}
                    \I \\ & \Sigma
                \end{smallmatrix}\!\!
            \end{pmatrix}
            \ttp {\begin{pmatrix}
                L & A
            \end{pmatrix}}
            = \tilde H.
        \end{aligned}
        \vspace{-1.6em}
    \end{equation}
\end{proof}
\vspace{0.2em}
There are two key parts to \Cref{cor:fact-update}:
First, \textit{if} there exists
an $\left(\!\begin{smallmatrix}
    \I \\ & \Sigma
\end{smallmatrix}\!\right)$-orthogonal transformation $\breve Q$ that reduces the matrix
$(L \;\; A)$ to a lower triangular matrix $(\tilde L \;\; 0)$, then this matrix $\tilde L$
is a Cholesky factor of the updated matrix $L\ttp L + A\Sigma\ttp A$. Second,
\Cref{prop:hyper-qr} confirms that such a transformation $\breve Q$ \textit{does}
exist if
the updated matrix is positive definite.
Moreover, its proof describes the construction of $\breve Q$ and $\tilde L$.

\subsection{Efficient representation and application of hyperbolic Householder transformations}

We now discuss some practical considerations regarding optimized software
routines based on (hyperbolic) Householder transformations:
{} 1. Applying an individual hyperbolic Householder transformation for each
    column limits the performance of numerical implementations. Instead we will consider a blocked
    variant that applies the product of multiple elementary hyperbolic Householder
    matrices at once.
{} 2. The choice of $\breve Q$ in \Cref{lem:hyper-householder} flips the
    sign of the first component of $x$. Since diagonal elements of
    Cholesky factors are conventionally taken to be positive, we will
    preserve their sign during the update by negating the first column of $\breve Q$ (cfr. \Cref{rem:sign-q}).
{} 3. The Householder reflection vectors $v$ from \eqref{eq:def-q-hyper-house-lem}
    will be normalized so their first element $v_1$ equals one.
    This simplifies the blocked variant,
    and allows the storage of $A$ to be reused for the reflection vectors.

Specifically, we represent the hyperbolic Householder matrix that annihilates the $i$-th row of $A$ as follows:
\newcommand\bi[1]{b_{\hspace{-1pt}(#1)}^{\!\phantom\top}}
\newcommand\taui[1]{\inv\tau_{\hspace{-1pt}(#1)}}
\newcommand\biT[1]{\ttp b_{\hspace{-1pt}(#1)}}
\begin{equation} \label{eq:def-q-i}
    \breve Q_{(i)} \defeq
    \scalebox{0.85}{$\begin{pmatrix}
        \I_{i-1} \\
        & \taui{i} - 1 & & -\taui{i} \biT{i} \\
        & & \I_{n - i} \\
        & \taui{i} \Sigma \bi{i} & & \I - \taui{i} \Sigma \bi{i} \biT{i}
    \end{pmatrix}$},
\end{equation}
with $\taui{i} = \frac{2v_1}{\ttp v \Ssign v}$ and $\left(\!\!\begin{smallmatrix}
    1 \\ b_{(i)}
\end{smallmatrix}\!\!\right) = \inv v_1 v$ for
$v$ defined in \eqref{eq:def-q-hyper-house-lem}.

The recursive nature of the proof of \Cref{prop:hyper-qr} motivates a
recursive expression for the product of the hyperbolic Householder matrices that is
constructed in \eqref{eq:product-q-proof-qr}.
\begin{lemma}[Product of hyperbolic Householder matrices] \label{lem:block-householder}
    The product of matrices $\breve Q_{(i)}$ as defined by \eqref{eq:def-q-i}
    is given by
    \begin{equation*}
        \begin{aligned}
            \breve Q_{(1)} \cdots \breve Q_{(k)}
            &=
            \scalebox{0.84}{$
            \begin{pmatrix}
                \inv T_{(k)} - \I & & -\inv T_{(k)} B_{(k)} \\
                & \I_{n-k-1} \\
                \Sigma \ttp B_{(k)} \inv T_{(k)} & & \I - \Sigma \ttp B_{(k)} \inv T_{(k)} B_{(k)}
            \end{pmatrix},
            $}
        \end{aligned}
    \end{equation*}
    where $
    \begin{aligned}[t]
        B_{(1)} &\defeq \biT{1},
        &T_{(1)} &\defeq \tau_{(1)}, \\
        B_{(i)} &\defeq \scalebox{0.84}{$\begin{pmatrix}
            B_{(i - 1)} \\
            \biT{i}
        \end{pmatrix}$}, & T_{(i)} &\defeq \scalebox{0.84}{$\begin{pmatrix}
            T_{(i - 1)} & B_{(i-1)} \Sigma \bi{i} \\
            & \tau_{(i)}
        \end{pmatrix}$}.
    \end{aligned}
    $
\end{lemma}
\begin{proof}
    By induction. The base case $\breve Q_{(1)}$ is trivial using \eqref{eq:def-q-i}.
    Then assume that \Cref{lem:block-householder} holds for factors 1 through $k$, and
    carry out the block matrix multiplication $\big(\breve Q_{(1)} \cdots \breve Q_{(k)}\big)\, \breve Q_{(k + 1)}$.
    We find that this product equals
    \begin{equation*}
        \begin{aligned}
            \scalebox{0.84}{$
            \left(\begin{array}{@{}c|c|c@{}}
                U_{(k+1)} - \I & & 
                -U_{(k+1)} B_{(k+1)}  \\\hline
                 & \I_{n-k-2} & \\\hline
                \Sigma \ttp B_{(k+1)} U_{(k+1)} & & \I - \Sigma \ttp B_{(k+1)} U_{(k+1)} B_{(k+1)}
            \end{array}\right),
            $}
        \end{aligned}
    \end{equation*}
    where we introduced $U_{(k+1)} \defeq 
        \scalebox{0.80}{$
        \begin{pmatrix}
            \inv T_{(k)} & -\inv T_{(k)} B_{(k)} \Sigma \bi{k+1} \taui{k+1} \\
            &  \taui{k+1}
        \end{pmatrix}$}$
    and
    $B_{(k+1)}
    $ as defined above.
    Observing that $U_{(k+1)} = \inv T_{(k+1)}$ completes the proof \cite[Fact~2.17.1]{bernstein_matrix_2009}.
\end{proof}
\begin{remark}
    The construction in \Cref{lem:block-householder}, using the product of the reflection vectors $B$ and a triangular matrix $U$,
    is based on the \textit{compact WY representation} of standard Householder matrices developed
    in \cite{schreiber_storage-efficient_1989}.
    The insight of using a recursion in terms of $T_{(i)}$ rather than
    $U_{(i)}$ was first published in \cite{puglisi_modification_1992}.
    Note that application of $\inv T_{(k)}$ is implemented
    using back substitution; no explicit inversion is performed.
\end{remark}

Thanks to this representation of hyperbolic Householder products,
application of $\breve Q$ to a given matrix requires mostly matrix--matrix
multiplications. This results in high arithmetic intensity, and optimal performance on modern processors. In contrast,
sequential application of individual Householder transformations is generally limited by memory bandwidth.

\section{Algorithms for factorization updates} \label{sec:algorithms}

The constructive proof of \Cref{prop:hyper-qr} can be leveraged to turn
\Cref{cor:fact-update} into a practical method for performing updates
of a given Cholesky factorization. This section presents
two core algorithms and later combines them into an efficient blocked factorization update procedure. 

\subsection{Two core algorithms, two perspectives}

Given a triangular matrix $L_{11} \in \Lower(\R^{n\times n})$,
an update matrix $A \in \R^{n\times m}$ and a diagonal matrix $\Sigma \in \Diag(\R^m)$,
\underline{\Cref{alg:stab-hyp-house-diag-block}} computes the
$\left( \!\begin{smallmatrix}
    \I \\ & \Sigma
\end{smallmatrix} \!\right)$-orthogonal
matrix $\breve Q$ and the updated lower triangular matrix
$\tilde L_{11} \in  \Lower(\R^{n \times n})$ such that, as in \Cref{cor:fact-update},
\begin{equation}
    \begin{aligned}
        \begin{pmatrix}
            \tilde L_{11} & 0
        \end{pmatrix} &= \begin{pmatrix}
            L_{11} & A_1
        \end{pmatrix} \breve Q.
    \end{aligned}
\end{equation}
In other words, given a Cholesky factor $L_{11}$ of $H$, \Cref{alg:stab-hyp-house-diag-block}
computes a Cholesky factor $\tilde L_{11}$ of $\tilde H \defeq H + A_1 \Sigma \ttp A_1$. \\[0.5em]
\underline{\Cref{alg:stab-hyp-house-tail-block}} uses \Cref{lem:block-householder} to apply the matrix $\breve Q$
generated by \Cref{alg:stab-hyp-house-diag-block} to matrices $L_{21} \in \R^{l\times n}$ and $A_2 \in \R^{l\times m}$:
\begin{equation}
    \begin{pmatrix}
        \tilde L_{21} & \tilde A_2
    \end{pmatrix} = \begin{pmatrix}
        L_{21} & A_2
    \end{pmatrix} \breve Q.
\end{equation}

\begin{algorithm2e}
    \small
    \caption{Block hyperbolic Householder update}
    \label{alg:stab-hyp-house-diag-block}
    \setstretch{1.2}
    \DontPrintSemicolon
    \KwIn{$L_{11} \in \Lower(\R^{n\times n})$, \hspace{0.3em} $A_1 \in \R^{n\times m}$, \hspace{0.3em} $\Sigma \in \Diag(\R^{m})$.}
    \KwOut{$\tilde L_{11} \in \Lower(\R^{n\times n})$, \hspace{0.3em} $B_{\phantom 1} \in \R^{n\times m}$, \hspace{0.3em} $T \in \Upper(\R^{n\times n})$.}
    \vspace{0.4em}
    \KwResult{\hspace{-3em}$
        \begin{aligned}[t]
            &\begin{pmatrix}
                \tilde L_{11} & 0
            \end{pmatrix} = \begin{pmatrix}
                L_{11} & A_1
            \end{pmatrix} \breve Q, \\
            \text{where } &\breve Q = \scalebox{0.75}{$\begin{pmatrix}
                \inv T - \I & -\inv T B \\
                \Sigma \ttp B \inv T & \I - \Sigma \ttp B \inv T B
            \end{pmatrix}$} \text{ is } \left(\!\begin{smallmatrix}
                \I \\ & \Sigma
            \end{smallmatrix}\!\right)\text{-orthogonal.}
        \end{aligned}
    $\\\vspace{0.3em}\rule{3.1em}{0.5pt}}
    \vspace{0.3em}
    Initialize
    $L_{(1)} \assign L_{11}$
    \hspace{3pt}
    and
    \hspace{3pt}
    $A_{(1)} \assign A_1$
    \;
    Initialize
    $\tilde L_{(0)} \assign (\,)$
    \hspace{8.4pt}
    and
    \hspace{3pt}
    $B_{(0)} \assign (\,)$
    \hspace{3pt}
    and
    \hspace{3pt}
    $T_{(0)} \assign (\,)$
    \;
    \For{$k = 1,2,\; ...\; n$}{
        \renewcommand*{\arraystretch}{1.08}
        Partition
        $\scalebox{0.85}{$\left(\begin{array}{@{}c|c@{}}
            \lambda \\\hline
            \ell & L_{(k+1)}
        \end{array}\right)$} \assign L_{(k)}$
        \hspace{3pt}and\hspace{3pt}
        $\scalebox{0.85}{$\left(\begin{array}{@{}c@{}}
            \ttp a \\\hline
            A^\prime
        \end{array}\right)$} \assign A_{(k)}$ \hspace{8cm}
        \phantom{m} \scalebox{0.9}{with $\lambda \in \R$, $a \in \R^m$}\;
        $\alpha^2 \assign \ttp a \Sigma a$, \quad
        $\tilde \lambda \assign \operatorname{sign}(\lambda) \sqrt{\lambda^2 + \alpha^2}$,
        \quad
        $\beta \assign \lambda + \tilde\lambda$\;\vspace{0.1em}
        $\taui{k} \assign \frac{2\beta^2}{\alpha^2 + \beta^2}$, \quad
        $\bi{k} \assign \inv\beta a$\mycomment{\textcolor{gray}{\small Construct $\breve Q_{(k)}$}}\vspace{0.3em}
        $w \assign \taui{k} \left( \ell + A^\prime \Sigma b_{(k)} \right)$\;
        $\left(\begin{array}{@{}c|c@{}}
            \tilde \ell & A_{(k+1)} \vphantom{\biT{k}}
        \end{array}\right) \assign \left(\begin{array}{@{}c|c@{\,}}
            w - \ell & A^\prime - w\,\biT{k}
        \end{array}\right)$\mycomment{\textcolor{gray}{\small Apply $\breve Q_{(k)}$}}
        $\tilde L_{(k)} \assign \scalebox{0.85}{\setstretch{1}$\left(\begin{array}{@{}c|c@{}}
             \\
            \tilde L_{(k-1)} & \tilde \lambda \\
            & \tilde \ell
        \end{array}\right)$}$\mycomment{\textcolor{gray}{\small Build updated Cholesky factor}}
        $B_{(k)} \assign \scalebox{0.85}{$\left(\begin{array}{@{}c@{}}
            B_{(k-1)} \\\hline
            \biT{k}
        \end{array}\right)$}$,\quad
        $T_{(k)} \assign \scalebox{0.85}{$\left(\begin{array}{@{}c|c@{}}
            {\hspace{1.8pt}}T_{(k-1)} & B_{(k-1)} \Sigma \bi{k} \\\hline
            & \tau_{(k)}
        \end{array}\right)$}$
        \;
    }
    $\tilde L_{11} = \tilde L_{(n)}$
    \hspace{3pt}
    and
    \hspace{3pt}
    $B = B_{(n)}$
    \hspace{3pt}
    and
    \hspace{3pt}
    $T = T_{(n)}$
    \;
\end{algorithm2e}
\begin{algorithm2e}
    \small
    \caption{Block hyperbolic Householder apply}
    \label{alg:stab-hyp-house-tail-block}
    \setstretch{1.2}
    \DontPrintSemicolon
    \KwIn{$L_{21} \in \R^{l\times n}$, \hspace{0.3em} $A_2 \in \R^{l\times m}$, \hspace{0.3em} $\Sigma \in \Diag(\R^{m})$,
    \phantom{$L_{21} \in \R^{l\times n}$, \hspace{0.4em}} $B_{\phantom 1} \in \R^{n\times m}$, \hspace{0.05em} $T \in \Upper(\R^{n\times n})$.}
    \KwOut{$\tilde L_{21} \in \R^{l\times n}$, \hspace{0.3em} $\tilde A_2 \in \R^{l\times m}$.}
    \vspace{0.4em}
    \KwResult{\hspace{-0.2em}$
        \begin{aligned}[t]
            &\begin{pmatrix}
                \tilde L_{21} & \tilde A_2
            \end{pmatrix} = \begin{pmatrix}
                L_{21} & A_2
            \end{pmatrix} \breve Q, \\
            &\text{with } \breve Q = \scalebox{0.75}{$\begin{pmatrix}
                \inv T - \I & -\inv T B \\
                \Sigma \ttp B \inv T & \I - \Sigma \ttp B \inv T B
            \end{pmatrix}$}.
        \end{aligned}
    $\\\vspace{0.3em}\rule{3.1em}{0.5pt}}
    \vspace{0.3em}
    $W \assign \left( L_{21} + A_2 \Sigma \ttp B \right)\inv T$\;
    $\tilde A_2 \assign A_2 - W B$\;
    $\tilde L_{21} \assign W - L_{21}$\;
\end{algorithm2e}

\begin{remark}[Two viewpoints of \Cref{alg:stab-hyp-house-diag-block}] \label{rem:viewpoints}
    One way to view \Cref{alg:stab-hyp-house-diag-block} is simply as a
    function that maps a given Cholesky factor and update matrix to an updated Cholesky factor, i.e.
    $(L,\, A,\, \Sigma) \mapsto \tilde L = \operatorname{chol}(L\ttp L + A\Sigma \ttp A)$.
    However,
    \Cref{alg:stab-hyp-house-diag-block}
    is more than a black box that performs the update: it also
    records all operations needed for the update in a linear transformation $\breve Q$
    that can be applied efficiently to other matrices as well (not just to $L$ and $A$).
    This is a powerful perspective, because it enables blocked variants
    of the Cholesky factorization update procedures that are both more efficient
    in practice, and that support the derivation of factorization updates for
    the Riccati recursion in \Cref{sec:riccati}.
\end{remark}

\subsection{Blocked algorithm} \label{sec:blocked-alg}

As suggested by \Cref{rem:viewpoints},
a combination of \Cref{alg:stab-hyp-house-diag-block,alg:stab-hyp-house-tail-block} can be used to develop a
\textit{blocked} variant of the factorization update procedure,
where the Cholesky factors and the update matrix are partitioned as follows:
\begin{equation}
    \left(\begin{array}{@{}c|c@{}}
        L & A
    \end{array}\right) =
    \scalebox{0.8}{$\left(\begin{array}{@{}cc|c@{}}
        L_{11} & & A_1 \\
        L_{21} & L_{22} & A_2
    \end{array}\right)$},
\end{equation}
with $L_{11}\in\Lower(\R^{r\times r}), A_1\in\R^{r\times m}$,
for some block size $r$.
The updated Cholesky factor $\tilde L = \operatorname{chol}(L\ttp L + A\Sigma \ttp A)$
can then be computed by iterating the following steps:
\begin{equation*}
    \begin{aligned}
        &0. \quad \text{Partition }
        \raisebox{1pt}{\scalebox{0.8}{$\left(\begin{array}{@{}cc|c@{}}
            L_{11} & & A_1 \\
            L_{21} & L_{22} & A_2
        \end{array}\right)$}}
        \assign
        \left(\begin{array}{@{}c|c@{}}
            L & A
        \end{array}\right); \\[-0.25em]
        &1. \quad \text{Find $\breve Q$ such that }\begin{pmatrix}
            \tilde L_{11} & 0
        \end{pmatrix} = \begin{pmatrix}
            L_{11} & A_1
        \end{pmatrix} \breve Q; \\
        &2. \quad \text{Compute }\begin{pmatrix}
            \tilde L_{21} & \tilde A_2
        \end{pmatrix} = \begin{pmatrix}
            L_{21} & A_2
        \end{pmatrix} \breve Q; \\
        &3. \quad \text{Go to 0. with } \left(\begin{array}{@{}c|c@{}}
            L & A
        \end{array}\right) \assign \left(\begin{array}{@{}c|c@{}}
            L_{22} & \tilde A_2
        \end{array}\right).
\end{aligned}
\end{equation*}
Each iteration updates the $r$ leftmost columns of $L$, and the procedure is
then applied recursively to the remainder $L_{22}$ until all columns of $L$
have been updated (at which point $L_{22}$ is empty).
\Cref{alg:stab-hyp-house-diag-block,alg:stab-hyp-house-tail-block} are used in
tandem to perform steps 1 and 2 respectively, ensuring that $\breve Q$ is
$\left( \!\begin{smallmatrix}
    \I \\ & \Sigma
\end{smallmatrix} \!\right)$-orthogonal. These two steps can be viewed as
a single multiplication $\left(\!\begin{smallmatrix}
    \tilde L_{11} & 0 \\
    \tilde L_{21} & \tilde A_2
\end{smallmatrix}\!\right) = \left(\!\begin{smallmatrix}
    L_{11} & A_1 \\
    L_{21}{}\vphantom{\tilde L_{11}} & A_2
\end{smallmatrix}\!\right) \breve Q$.
The proof for this blocked algorithm is a straightforward generalization of the proof of \Cref{prop:hyper-qr}.
\footnote{Apply Prop.\,\ref{prop:hyper-qr} to $(L_{11} \,|\, A_1)$ instead of applying Lem.\,\ref{lem:hyper-householder} to $(\lambda, a)$.}.

In the case where $\Sigma=\pm\I$, this algorithm requires
$mn^2 + \tfrac{r-1}4 n^2 + \tfrac{r-1}2 mn - \tfrac{r^2 - 3r + 2}{4}n$ fused multiply-add operations,
$\tfrac12 n^2 + mn + \tfrac12nr + \landauO(n)$ multiplications, $\tfrac12 n^2 + \landauO(n)$ additions, $2n$ divisions, and $n$ square roots.
For $r=1$, this amounts to $mn^2 + n^2 + mn + \landauO(n)$ operations in total.

\section{Application to the Riccati recursion} \label{sec:riccati}

As a relevant application of Cholesky factorization updates, we consider the
solution of the structured Newton systems that arise when applying an augmented
Lagrangian method (ALM) such as \QPALMOCP{} to a linear-quadratic optimal control
problem (OCP) with inequality constraints, as described in \cite{lowenstein_qpalm-ocp_2024}.
We reuse the OCP definition from \cite[Eq.~1]{lowenstein_qpalm-ocp_2024}, and define the Jacobians of the dynamics $F_j \defeq {\begin{pmatrix}
    B_j & A_j
\end{pmatrix}}$, and of the inequality constraints $G_j \defeq {\begin{pmatrix}
    D_j & C_j
\end{pmatrix}}$ and $G_N \defeq C_N$. The Hessian matrices of the stage cost functions $l_j(u^j, x^j)$
are denoted by $H^l_j \defeq \left(\!\begin{smallmatrix}
    R_j & S_j \\ \ttp S_j & Q_j
\end{smallmatrix}\!\right)$, and $H^l_N \defeq Q_N$.

In each inner iteration of \QPALMOCP{}, the Newton system
\cite[Eq.~16]{lowenstein_qpalm-ocp_2024} needs to be solved at the current iterate $x$.
This is equivalent to solving the following equality constrained quadratic program
with optimal control structure:
\def\dx{\Delta\hspace{-0.2pt}x}
\def\du{\Delta\hspace{-0.8pt}u}
\begin{equation} \label{eq:ocp-qpalm-inner}
    \hspace{-1em}
    \begin{aligned}
        &\minimize_{\du,\dx} && \textstyle \sum_{j=0}^{N-1}\; l^{\setJ(x)}_j\!(\du^j, \dx^j) \;+\; l^{\setJ(x)}_N(\dx^N) \\[-0.3em]
        &\subjto && \dx^{j+1} = A_j \dx^j + B_j \du^j + e^j(x), \\
        &&& \dx^0 = x_{\text{init}} - x^0,
    \end{aligned} \hspace{-2em}
\end{equation}
where the Newton step is given by the minimizer $(\du,\dx)$.
The cost functions $l^{\setJ(x)}_j$ are quadratic
with
$\hess l^{\setJ(x)}_j = H^\setJ_j\!(x) \defeq H^l_j + \ttp G_j \Sigma^{\setJ(x)}_{j} G_j$.
\cite[Eq.~18]{lowenstein_qpalm-ocp_2024}.
The diagonal ALM penalty matrices $\Sigma^{\setJ(x)}_{j}$ depend on the active set $\setJ(x)$ \cite[Eq.~14]{lowenstein_qpalm-ocp_2024}
at the current iterate $x$: entries $(\Sigma^{\setJ(x)})_{ii}$ are set to a positive penalty parameter if they correspond to an active inequality constraint, and zero otherwise.
Assuming that only a small number of constraints enter or leave the active set between successive iterates $x$ and $\tilde x$, the matrices
$H^\setJ_j$ undergo a low-rank update $H^\setJ_j\!(\tilde x) = H^\setJ_j\!(x) + \ttp G_j (\tilde\Sigma^{\setJ}_{j} - \Sigma^{\setJ}_{j}) G_j$,
where the penalty matrices $\Sigma^{\setJ}_{j} \defeq \Sigma^{\setJ(x)}_{j}$ and $\tilde\Sigma^{\setJ}_{j} \defeq \Sigma^{\setJ(\tilde x)}_{j}$ differ in a small number of entries only.
This motivates using update routines 
to avoid full factorizations
at each iteration.

The original implementation of \QPALMOCP{} uses a Schur complement method for
solving the Newton system, with factorization updates based on \cite{gill_methods_1974} and the Sherman--Morrison formula.
In the present paper, we instead consider the well-known Riccati recursion for solving problems of the form \eqref{eq:ocp-qpalm-inner}.
We make use of the factorized variant proposed in \cite{frison_efficient_2013}: the factorization procedure
is given in \Cref{alg:fact-riccati}, with some ALM-specific modifications. Using this factorization,
the Newton step can be recovered as described in \cite[Alg.~2]{frison_efficient_2013}.

We now present \underline{\Cref{alg:fact-upd-riccati}}, a novel method for efficiently updating
the Cholesky factors computed by \Cref{alg:fact-riccati}:
Given the factors $\LPuu{j}, \LPxu{j}, \LPxx{j}$ of matrices $H_j$ computed by \Cref{alg:fact-riccati}
for the iterate $x$ with penalties $\Sigma^\setJ_j$, \Cref{alg:fact-upd-riccati} obtains the updated factors
$\tilLPuu{j}, \tilLPxu{j}, \tilLPxx{j}$ that can be used to compute the Newton step at a new
iterate $\tilde x$ with penalties $\tilde\Sigma^\setJ_j$. Imporantly, it does so without actually executing \Cref{alg:fact-riccati} at $\tilde x$,
resulting in a lower run time.

\begin{proof}
Denote by $\tilde H_j$ the matrices that would be computed on \cref{ln:H-riccati} of \Cref{alg:fact-riccati},
if applied at $\tilde x$ with $\tilde\Sigma^\setJ_j$.  We will show inductively that
\Cref{alg:fact-upd-riccati} computes the factors $\left(\!\begin{smallmatrix}
\tilLPuu{j} \\ \tilLPxu{j} & \tilLPxx{j}
\end{smallmatrix}\!\right) \defeq \operatorname{chol}(\tilde H_j)$:
As the induction hypothesis, we state that the cost-to-go matrices $P_j = \LPxx{j}\LPTxx{j}$ undergo a low-rank update between $x$ and $\tilde x$:
$\tilde P_{j} = \tilLPxx{j}\tilLPTxx{j} = P_{j} + \Phi_{j} \Ssign_{j} \ttp \Phi_{j}$ for some $\Phi_j$ and $\Ssign_j$.
At the final stage $N$, we have $\tilde P_N = P_N + \ttp G_N (\tilde\Sigma^{\setJ}_N - \Sigma^{\setJ}_N) G_N$,
covering the base case with $\Phi_N = \ttp G_N$ and $\Ssign_N = \tilde\Sigma^{\setJ}_{N} - \Sigma^{\setJ}_{N}$.
This update is readily performed by \Cref{alg:stab-hyp-house-diag-block} to obtain
$\tilLPxx{N}$.
For all previous stages $j\lt N$, there are two terms in the definition of matrix $H_j$ (Alg.\,\ref{alg:fact-riccati}, \cref{ln:H-riccati}) that may change between $x$ and $\tilde x$:
the second term due to changes to the penalty matrix $\tilde\Sigma^{\setJ}_{j}$ in the current stage,
and the third term due to changes to the cost-to-go matrix
$\tilde P_{j+1} = \tilLPxx{j+1}\tilLPTxx{j+1}$ in the next stage.
We have that $\tilde H_j = H_j + \ttp G_j (\tilde\Sigma^\setJ_j - \Sigma^\setJ_j) G_j + \ttp F_j (\tilde P_{j+1} - P_{j+1}) F_j$.
Using the induction hypothesis
$\tilde P_{j+1} = P_{j+1} + \Phi_{j+1} \Ssign_{j+1} \ttp \Phi_{j+1}$,
we have $\tilde H_j = H_j + \Upsilon_j \Ssign_j \ttp \Upsilon_j$, with $\Upsilon_j \defeq \left(\ttp F_j \Phi_{j+1}\;\;\; \ttp G_j\right)$
and $\Ssign_j \defeq \operatorname{blkdiag}(\Ssign_{j+1},\; \tilde\Sigma^\setJ_j - \Sigma^\setJ_j)$.
To compute the Cholesky factorization of $\tilde H_j$, we rely on the blocked algorithm outlined in \Cref{sec:blocked-alg},
as written out in detail on \cref{ln:riccati-upd-uu,ln:riccati-upd-xu,ln:riccati-upd-xx} of \Cref{alg:fact-upd-riccati}.
Finally, \cref{ln:riccati-upd-xx} implies that
$\tilde P_j = \tilLPxx{j}\tilLPTxx{j} = \LPxx{j}\LPTxx{j} + \Phi_j \Ssign_j \ttp \Phi_j = P_j + \Phi_j \Ssign_j \ttp \Phi_j$,
completing the proof.
\end{proof}

\begin{algorithm2e}
    \small
    \caption{Riccati factorization \cite[Alg.~3]{frison_efficient_2013}}
    \label{alg:fact-riccati}
    \setstretch{1.12}
    \DontPrintSemicolon
    \KwIn{\hspace{-1pt}%
        $\{H^l_j\}_{j=0}^{N-1}, Q_N$ (Hessians of OCP cost), \\
        $\{F_j\}_{j=0}^{N-1}, \{G_j\}_{j=0}^N$ (dynamics \& constraint Jacobians), \\
        $\{\Sigma^\setJ_j\}_{j=0}^N$ (ALM penalty factors).
        }
    \KwOut{\hspace{-1pt}%
        $\{\LPuu{j}\}_{j=0}^{N-1}, \{\LPxu{j}\}_{j=0}^{N-1}, \{\LPxx{j}\}_{j=0}^N$ \\
        \quad(Cholesky factors).}
    \vspace{0.4em}
        $P_N \assign Q_N + \ttp G_N \Sigma^\setJ_N G_N$\;
        $\LPxx{N} \assign \operatorname{chol}(P_N)$\;
        \For{$j = N-1, N-2,\; ...\; 0$}{
            $H_j \assign H^l_j + \ttp G_j \Sigma^{\setJ}_{j} G_j + \ttp F_j\, \LPxx{j+1} \LPTxx{j+1}\,F_j$\;\label{ln:H-riccati}
            \setstretch{0.9}$\begin{pmatrix}
                \LPuu{j} \\ \LPxu{j} & \LPxx{j}
            \end{pmatrix} \assign \operatorname{chol}(H_j)$\;\label{ln:chol-H-riccati}
        }
\end{algorithm2e}
\begin{algorithm2e}
    \small
    \caption{Riccati factorization updates}
    \label{alg:fact-upd-riccati}
    \setstretch{1.12}
    \DontPrintSemicolon
    \KwIn{\hspace{-1pt}%
        $\{F_j\}_{j=0}^{N-1}, \{G_j\}_{j=0}^N$ (dynamics \& constraint Jacobians), \\
        $\{\Sigma^\setJ_j\}_{j=0}^N, \{\tilde\Sigma^\setJ_j\}_{j=0}^N$ (old and new penalty factors), \\
        $\{\LPuu{j}\}_{j=0}^{N-1}, \{\LPxu{j}\}_{j=0}^{N-1}, \{\LPxx{j}\}_{j=0}^N$ \\
        \quad(old Cholesky factors).
        }
    \KwOut{\hspace{-1pt}%
        $\{\tilLPuu{j}\}_{j=0}^{N-1}, \{\tilLPxu{j}\}_{j=0}^{N-1}, \{\tilLPxx{j}\}_{j=0}^N$ \\
        \quad(updated Cholesky factors).}
    \vspace{0.4em}
        $\Phi_N \assign \ttp G_N$,\qquad
        $\Ssign_N \assign \tilde\Sigma^\setJ_N - \Sigma^\setJ_N$\;
        $\begin{pmatrix}
            \tilLPxx{N} & 0
        \end{pmatrix} \assign \begin{pmatrix}
            \LPxx{N} & \Phi_N
        \end{pmatrix} \Qx{N}$
        using Alg.\,\ref{alg:stab-hyp-house-diag-block}, \hspace{8cm}
        \phantom{m} where $\Qx{N}$ is $\left(\!\begin{smallmatrix}
            \I \\ & \Ssign_N
        \end{smallmatrix}\!\right)$-orthogonal\;
        \For{$j = N-1, N-2,\; ...\; 0$}{
            $\begin{pmatrix}
                \Upsu{j} \\ \Upsx{j}
            \end{pmatrix} \assign \begin{pmatrix}
                \ttp F_j \Phi_{j+1} & \!\ttp G_j
            \end{pmatrix}$,
            \quad
            $\Ssign_j \assign \left(\!\begin{smallmatrix}
                \Ssign_{j+1} \\ & \tilde\Sigma^\setJ_j - \Sigma^\setJ_j
            \end{smallmatrix}\right)$\;
            $\begin{pmatrix}
                \tilLPuu{j} & 0{\hspace{1.8pt}}
            \end{pmatrix} \assign \begin{pmatrix}
                \LPuu{j} & \Upsu{j}
            \end{pmatrix} \Qu{j}$
            using Alg.\,\ref{alg:stab-hyp-house-diag-block}, \hspace{8cm}
            \phantom{m} where $\Qu{j}$ is $\left(\!\begin{smallmatrix}
                \I \\ & \Ssign_j
            \end{smallmatrix}\!\right)$-orthogonal\;\label{ln:riccati-upd-uu}
            $\begin{pmatrix}
                \tilLPxu{j} & \hspace{-4pt}\Phi_j
            \end{pmatrix} \assign \begin{pmatrix}
                \LPxu{j} & \Upsx{j}
            \end{pmatrix} \Qu{j}$
            using Alg.\,\ref{alg:stab-hyp-house-tail-block}\;\label{ln:riccati-upd-xu}
            $\begin{pmatrix}
                \tilLPxx{j} & {0\hspace{2.8pt}}
            \end{pmatrix} \assign \begin{pmatrix}
                \LPxx{j} & {\hspace{2.2pt}} \Phi_j {\hspace{3pt}}
            \end{pmatrix} \Qx{j}$
            using Alg.\,\ref{alg:stab-hyp-house-diag-block}, \hspace{8cm}
            \phantom{m} where $\Qx{j}$ is $\left(\!\begin{smallmatrix}
                \I \\ & \Ssign_j
            \end{smallmatrix}\!\right)$-orthogonal\;\label{ln:riccati-upd-xx}
        }
\end{algorithm2e}

\section{Performance of optimized updating routines} \label{sec:performance}

\begin{figure}
    \centering
    \includegraphics[width=83mm, clip, trim=2.5mm 3.5mm 3.7mm 3mm]{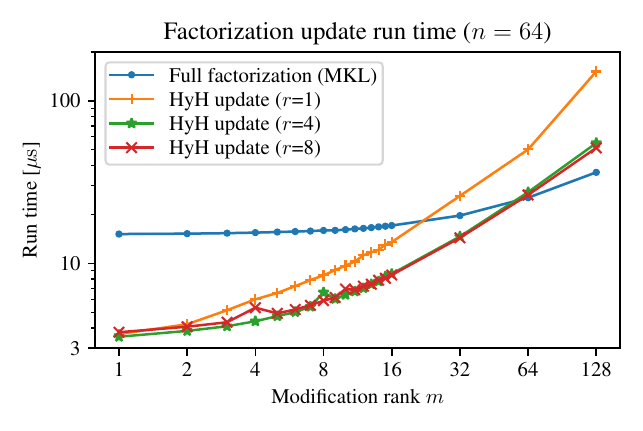}
    \caption{Absolute run time of different factorization update methods, for $L\in \R^{64\times64}, A\in \R^{64\times m}$.
    \textit{Full factorization} refers to adding $A\Sigma \ttp A$ to the previous matrix $H$ (\texttt{dsyrk})
    and performing a fresh Cholesky factorization (\texttt{dpotrf}).\\
    \textit{HyH update} refers to the blocked hyperbolic Householder-based algorithm from \S\ref{sec:blocked-alg}, for various block sizes $r$.\vspace{-0.3em}}
    \label{fig:shh-timings}
\end{figure}
The blocked algorithm from \Cref{sec:blocked-alg} has been implemented in \Cpp{},
with highly optimized and vectorized micro-kernels for \Cref{alg:stab-hyp-house-diag-block,alg:stab-hyp-house-tail-block}.
The source code has been made available on {\small \url{https://github.com/kul-optec/hyhound}}.
\Cref{fig:shh-timings} compares the run times of this implementation to those of a full factorization without low-rank updates, for various ranks and using different block sizes.
\footnote{Experiments were carried out on an Intel Core i7-11700 at 2.5 GHz (fixed), using v2025.0 of the Intel MKL.}

For small modification ranks $m$, the hyperbolic Householder factorization updates are over four times faster than a full factorization.
For $m>2$, the blocked variants ($r$\,=\,4,\,8) are significantly faster than the non-blocked variant ($r$\,=\,1).
This motivates the use of blocked hyperbolic Householder transformations over the column-wise methods from \cite{gill_methods_1974}.

\begin{figure}
    \centering
    \includegraphics[width=83mm, clip, trim=2.5mm 3.5mm 3.7mm 3mm]{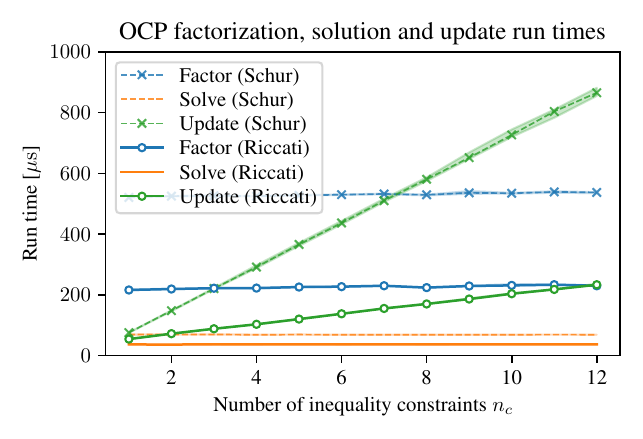}
    \caption{Run time comparison of the individual factorization, solution, and factorization update steps when solving \eqref{eq:ocp-qpalm-inner},
    for the Schur complement method versus the Riccati recursion from \S\ref{sec:riccati}.
    OCP dimensions: $N\!=\!24,n_x\!=\!24, n_u\!=\!8$.}
    \label{fig:ocp-timings}
\end{figure}

Finally, we compare the performance of the Schur complement method from \cite{lowenstein_qpalm-ocp_2024} and
the Riccati recursion algorithms presented in \Cref{sec:riccati} when applied to the semismooth Newton
system \eqref{eq:ocp-qpalm-inner} of \QPALMOCP{}. The system is solved for randomly generated benchmark OCPs with varying numbers of
inequality constraints $n_c$. As a benchmark, the penalty factors of $25\%$ of the inequality constraints (selected at random) are changed between $x$
and $\tilde x$, resulting in larger modification ranks as the number of constraints increases.

\Cref{fig:ocp-timings} demonstrates that the Riccati-based solver significantly outperforms the Schur complement approach.
In particular, when combined with the proposed factorization update procedure, it provides an efficient method for solving successive Newton systems with an optimal control structure.

\bibliographystyle{IEEEtranS}
\bibliography{QPALM-OCP}

\end{document}